\documentclass{amsart}    

\usepackage{graphicx}

\usepackage{amssymb}
\usepackage{amsmath} 
\usepackage{amscd} 
\usepackage[all]{xypic}
\usepackage{url}
\usepackage{color}

\newcommand{\N}{\mathcal{N}} 

\newcommand{\fg}{\mathfrak g}

\newcommand{\fp}{\mathfrak p}

\newcommand{\fn}{\mathfrak n}
\newcommand{\Ad}{{\rm Ad}}

\newcommand{\R}{\mathbb{R}}

\newtheorem{prop*}{Proposition}

\newtheorem{thm*}{Theorem}
\newtheorem{lemma*}{Lemma}

\newtheorem{cor*}{Corollary}

\newtheorem{rem*}{Remark}
\newtheorem{def*}{Definition}

\makeatletter
\@namedef{subjclassname@2020}{%
  \textup{2020} Mathematics Subject Classification}
\makeatother

\begin{document}

\title[]{
On symmetries of a sub--Riemannian structure with growth vector $(4,7)$
}

\keywords{
Nilpotent algebras, Lie symmetry group, Carnot groups,  sub--Riemannian geodesics}

\subjclass[2020]{53C17, 22E60, 35R03}

\author{
Jaroslav Hrdina, Ale\v s N\' avrat and Lenka Zalabov\'a
}

\address{
JH, AN: Institute of Mathematics, 
	Faculty of Mechanical Engineering,  Brno University of Technology,
	Technick\' a 2896/2, 616 69 Brno, Czech Republic; LZ: Institute of Mathematics, 
	Faculty of Science, University of South Bohemia, 
	Brani\v sovsk\' a 1760, 370 05 \v Cesk\' e Bud\v ejovice, and
    Department of Mathematics and Statistics, 
	Faculty of Science, Masaryk University,
	Kotl\' a\v rsk\' a 2, 611 37 Brno, Czech Republic
}

\email{hrdina@fme.vutbr.cz, navrat.a@fme.vutbr.cz,lzalabova@gmail.com
}	

\thanks{
The first and second authors was supported by the grant no. FSI-S-20-6187. Third author is supported by grant no. 20-11473S Symmetry and invariance in analysis, geometric modeling and control theory from the Czech Science Foundation. 
We thank to Luca Rizzi for  useful discussions during Winter School Geometry and Physics, Srn\'i, 2020. Finally, we thank the referee for valuable comments.
}

\maketitle

\thispagestyle{empty}

\vspace{7pt}

\begin{abstract}
We study symmetries of specific left--invariant sub--Riemannian structure with filtration $(4,7)$ and their impact on sub--Riemannian geodesics of corresponding control problem. We show that there are two very different types of geodesics, they either do not intersect the fixed point set of symmetries or are contained in this set for all times. We use the symmetry reduction to study properties of geodesics.
\end{abstract}

\section{Introduction}

Symmetries of geometric structures play an important role in differential geometry and geometric control theory. Indeed, the existence of big amount symmetries or the existence of a special symmetry of the geometric structure often induces restrictions on its properties like the curvature etc. In particular, if the symmetry group acts transitively, the space is homogeneous and one can read off many properties just by restricting to one point, \cite{CS,zel}. 
Moreover, in geometric control theory, symmetries of control systems and their fixed points can be used for finding of distinguished points of geodesics like cut or cusp points, \cite{mya1,Rizzi,mope}. 
In this paper, we focus on the role of symmetries and their fixed points for specific filtration with the growth vector $(4,7)$ and their impact on special geodesics of the corresponding sub--Riemannian structure.

The motivation from applications comes from  \cite{hz} where the first and third authors study local control of a planar mechanism with $7$--dimensional configuration space. The robot in question consist of a root block in the shape of an equilateral triangle together with three branches that have passive wheels at their ends, where each of the branches is connected to one vertex of the root block via prismatic joint and one of the joints is simultaneously revolute joint, see Figure \ref{t}. 

 \begin{figure}[h]

 \includegraphics[height=40mm]{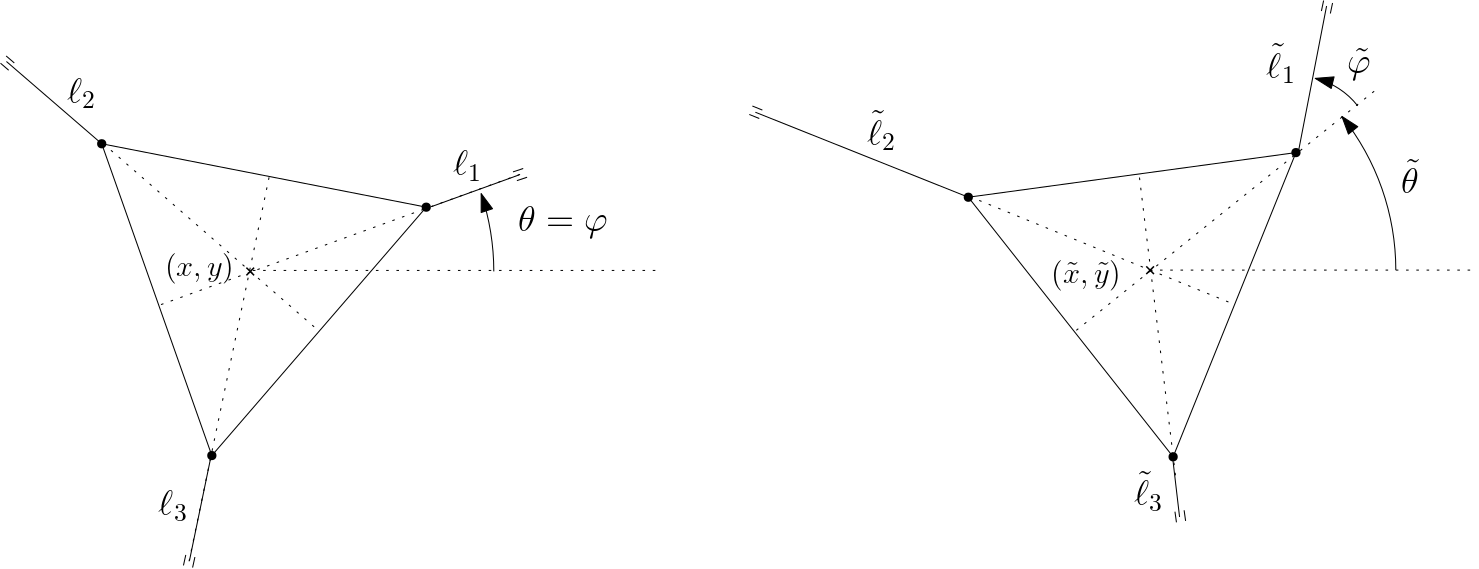}
 \caption{Robot motion in the configuration space with local coordinates  
 $(x,\ell_1,\ell_2,\ell_3,y,\theta, \varphi)$ 
 }
 \label{t}
 \end{figure}

Under the assumption that the robot moves with no slipping nor sliding, one derives three non--holonomic condition of the motion, one for each branch, and these determine $4$--dimensional distribution of admissible directions given (locally) on the configuration space. The choice of a sub--Riemannian metric allows to study (local)
optimal control, however, corresponding control problem is highly non--linear and hard to solve. Nevertheless,  it is sufficient (locally) to swap to its nilpotent approximation $N$, \cite{b96,J}.

Denoting by $N_0,N_1,N_2,N_3$ generators of the $4$--distribution $\N$ on $N$, it turns out that  the only non--trivial Lie brackets are the brackets $N_{01}:=[N_0,N_1]$, $N_{02}:=[N_0,N_2]$ and $N_{03}:=[N_0,N_3]$, so we get a nilpotent Lie algebra $\fn$.
In particular, these brackets do not belong to $\N$ and $(N,\N)$ is a Carnot group with filtration $(4,7)$. 
Altogether, we get the flat distribution $(N,\N)$ of constant type $\fn$.
Moreover, the choice of generators allows us to consider the decomposition
$E+V=\langle N_0\rangle+\langle N_1,N_2,N_3 \rangle$ 
of $\N$ into one--dimensional and three--dimensional involutive distributions
and compatible sub--Riemannian metric $g$ by declaring the vectors $N_i,i=0,1,2,3$ orthonormal. 
In fact, $(N,\N,E+V)$ can be viewed as a flat model of so--called generalized path geometry, \cite{CS}. Alongside, it turns out that the sub--Riemannian structure $(N,\N,g)$ is a flat structure of constant type $(\fn,\frak{so}(3,\R))$, \cite{zel}, i.e. the metric $g$ is invariant with respect to  the action of suitable $SO(3,\R)$. We describe these structures and their symmetries in detail in Section \ref{sec2}.

In Section \ref{sec-control}, we study control problem corresponding to the sub--Riemannian structure in question. We apply Hamiltonian concepts to approach this control problem, \cite{ABB}. In particular, we describe control functions and normal geodesics of the problem in detail (and strictly abnormal geodesics cannot appear for $1$--step filtrations, so we speak just about geodesics).
Let us remind that the set of points where geodesics intersect each other and the corresponding geodesic segments have equal length is called the Maxwell set. Conjungate points are defined as critical points of the exponential map. It is proved that the normal extremal trajectory that does not contain pieces of abnormal geodesics loses its optimality in the conjungate point or in the Maxwell point, \cite{ABB}. 

In many cases, Maxwell set contains sets of fixed points of symmetries. Indeed, if a geodesic meets a fixed point of a symmetry, then the action of the symmetry can give such set of geodesics, \cite{Rizzi, talianky, mya1}. We show in Section \ref{sec4} that this is not the case of our filtration. 
In particular, we study relations of geodesics and fixed--point set of symmetries. We show in Theorem \ref{thm1} that each geodesic starting at the origin either do not intersect the fixed--point set or is contained in this set for all times. Thus geodesics are of two very different types. We use the symmetry reduction to study geodesics contained in the fixed-point set. In particular, we relate these geodesics to geodesics in the Heisenberg group to find their cut--time in Theorem \ref{thm2}.

\section{Model Carnot group equipped with filtration $(4,7)$} \label{sec2}
Let us consider 
coordinates $(x,\ell_1,\ell_2,\ell_3,y_1,y_2,y_3)$ of vector space
$\R^7 \cong \R^4 \oplus \R^3 $
and model vector fields
\begin{align}
\begin{split}
&N_0=\partial_x-{\ell_1 \over 2}{\partial_{y_1}}-{\ell_2 \over 2}{\partial_{y_2}}-{\ell_3 \over 2}{\partial_{y_3}},\\ 
&N_1=\partial_{\ell_1}+{x \over 2} \partial_{y_1}, \ \ 
N_2=\partial_{\ell_2}+{x \over 2} \partial_{y_2}, \ \ 
N_3=\partial_{\ell_3}+{x \over 2} \partial_{y_3}, 
 \label{vf}
\end{split}
\end{align}
where the symbol $\partial$ stands for partial derivative. Let us note that these fields are precisely symmetric model vector fields introduced in \cite[Section 7.5.1]{ABB}. The only non--trivial Lie brackets are 
\begin{align}\label{n010203}
\begin{split}
N_{01}=[N_0,N_1]= \partial_{y_1}, \ \ N_{02}=[N_0,N_2]= \partial_{y_2}, \ \ N_{03}=[N_0,N_3]= \partial_{y_3}
\end{split}.
\end{align}
 The fields \eqref{vf}  and \eqref{n010203} then determine a  $2$--step nilpotent Lie algebra $\fn$.

\begin{rem*} \label{rem1}
Let us remark that each triple 
$N_0, N_h \in \langle N_1,N_2,N_3 \rangle$  and $[N_0,N_h]$ 
form a $3$--dimensional Heisenberg subalgebra in the Lie algebra $\fn$. Thus $\fn$ can be naively viewed as a `bunch' of  Heisenberg algebras.
\end{rem*}

The Lie algebra $\fn$ corresponds to  a Carnot group $N$ such that the fields $N_0,$ $N_1,$ $N_2,$ $N_3,$ $N_{01},$ $N_{02},$ $N_{03}$ are left--invariant for the corresponding group structure. 
We can compute this group structure just by taking the flows. 
Under identification of a point 
$g_1=(x,\ell_1,\ell_2,\ell_3,y_1,y_2,y_3) \in N$ with  image of the exponential map $g_1=\exp(x N_0 + \ell_1 N_1 + \ell_1 N_2+ \ell_1 N_3 
+ y_{1} N_{01} + y_{2} N_{02}+ y_{03} N_{03}  ) (o)$
the product is given by 
$$g_1g_2=\exp(t(x N_0 + \ell_1 N_1 + \ell_1 N_2+ \ell_1 N_3 
+ y_{1} N_{01} + y_{2} N_{02}+ y_{03} N_{03} )) (g_2) $$
evaluated in $t=1$. Using the Lie algebra structure \eqref{n010203}
 the group structure on~$N =  \mathbb R^4 \oplus \mathbb R^3$ reads as follows
\begin{align}\label{grupa}
\begin{pmatrix} 
x \\ \ell_1 \\ \ell_2\\ \ell_3\\ y_1\\ y_2\\ y_3
\end{pmatrix}
\cdot 
 \begin{pmatrix} 
 \tilde x \\ \tilde \ell_1 \\ \tilde \ell_2\\ \tilde \ell_3\\ \tilde y_1\\ \tilde y_2\\ \tilde y_3
 \end{pmatrix}
 =
 \begin{pmatrix} 
x+  \tilde x \\ \ell_1 + \tilde \ell_1 \\\ell_2 + \tilde \ell_2\\\ell_3+ \tilde \ell_3\\
y_1+ \tilde y_1 + \frac{1}{2} (x \tilde \ell_1 - \tilde x \ell_1)\\y_2+ \tilde y_2 + \frac{1}{2} (x \tilde \ell_2 - \tilde x \ell_2)\\y_3+ \tilde y_3 + \frac{1}{2} (x \tilde \ell_3 - \tilde x \ell_3)
  \end{pmatrix}.
 \end{align}

In particular, $\N=\langle N_0,N_1,N_2,N_3\rangle$ forms a $4$--dimensional left--invariant distribution on $N$. Moreover, our choice allows us to consider the decomposition
\begin{align}
\label{decom}
\N=\langle N_0 \rangle \oplus \langle N_1, N_2, N_3 \rangle 
\end{align}
of $\N$ into $1$--dimensional distribution and $3$--dimensional involutive distribution,  both left--invariant. Then by declaring $N_0$, $N_1$, $N_2$, $N_3$ orthonormal we define compatible sub--Riemannian metric $g$ on $\N$. Altogether, we get left--invariant sub--Riemannian 
structure $(N,\N,g)$ which is related to the left--invariant optimal control problem
 written in coordinates $(x, \ell_1,\ell_2,\ell_3,y_1,y_2,y_3)$
as   

\begin{align} \label{control}
\dot q (t)
=u_0\left(
\begin{smallmatrix}
1\\
0\\
0\\
0\\
-{\ell_1 \over 2}\\
-{\ell_2 \over 2}\\
-{\ell_3 \over 2}
\end{smallmatrix}
\right) 
+u_1 \left(
\begin{smallmatrix}
0\\
1\\
0\\
0\\
{x \over 2}\\
0\\
0
\end{smallmatrix}
\right) 
+u_2 \left(
\begin{smallmatrix}
0\\
0\\
1\\
0\\
0\\
{x \over 2}\\
0
\end{smallmatrix}
\right) 
+u_3 \left(
\begin{smallmatrix}
0\\
0\\
0\\
1\\
0\\
0\\
{x \over 2}
\end{smallmatrix}
\right) 
\end{align}
for $t>0$ and $q$ in $N$ and the control $u=(u_0,u_1,u_2,u_3) \in \R^4$ with the boundary condition $q(0)=q_1,$ $q(T)=q_2$ for fixed points $q_1, q_2 \in N$,
where we minimize 
\begin{align} \label{energy}
{1 \over 2}\int_0^T (u_0^2+u_1^2+u_2^2+u_3^2) dt.
\end{align}
Symmetries of the left--invariant sub--Riemannian structure $(N,\N, g)$, i.e. symmetries of the control system (\ref{control},\ref{energy}),
are automorphisms of $N$ preserving the distribution $\N$ and sub--Riemannian metric $g$. 
They form a finite--dimensional Lie group and we can describe its Lie algebra of infinitesimal symmetries using Cartan--Tanaka theory since we deal with flat distribution, \cite{zel,ams}.

Let us view $\fn=\fn_{-1} \oplus \fn_{-2}$ as an abstract Lie algebra with $\fn_{-1}$ spanned by $e_i=N_i(o)$, $i=0,..,3$ and $\fn_{-2}$ spanned by $e_{i+4}=N_{0i}(o)$, $i=1,..,3$. Here $o$ denotes the origin, i.e. identity element. Then the distribution $\N$ corresponds to the subspace $\fn_{-1} \subset \fn$ and $g$ corresponds to a $2$--tensor $g(o)$ defined on $\fn_{-1}$. 
We define
$\fn_0 \subset \frak{so}(\fn_{-1})$ 
to be the Lie algebra of the Lie group of all automorphisms of the graded
nilpotent algebra $\fn$ preserving the metric $g(o)$ on $\fn_{-1}$ , i.e. the algebra of
certain derivations of $\fn$.
Here the action of automorphisms of $\fn_0$ on $\fn_{-1}$ is exactly the adjoint action. Let us discuss explicitly the action and corresponding reduction.

\begin{lemma*} \label{akce}
The algebra $\fn_0$ of metric preserving derivations of $\fn$ equals to  $\frak{so}(3,\R)$.
\end{lemma*}
\begin{proof}
The algebra $\frak{so}(\fn_{-1}) \simeq \frak{so}(4,\R)$ is spanned by $6$ elements $e_{ij}$, $i,j=0,1,2,3$, $i<j$, each of which generates $\frak{so}(2,\R)$, so the action of $e_{ij}$ on generators of $\fn_{-1}$ takes form $[e_{ij},e_i]=e_j$ and  $[e_{ij},e_j]=-e_i$.
Let us now discuss the compatibility of this action with the Lie bracket on the whole $\fn$. The only non--vanishing Lie brackets are $[e_0,e_i]=e_{i+4}$, $i=1,2,3$, and thus $[e_j,e_k]=0$ for $j,k=1,2,3$. With the help of Jacobi identity, we compute
$$ 0=[e_{0i},[e_j,e_k]]
=[[e_{0i},e_j],e_k]+[e_j,[e_{0i},e_k]]
$$
for $i,j,k=1,2,3$.
Then for $i=j$ we get
$$ 0=[[e_{0i},e_i],e_k]+[e_i,[e_{0i},e_k]]=[-e_0,e_k] = -e_{4+k}
$$
which is a contradiction. Thus elements $e_{0i}$, $i=1,2,3$, cannot appear and $\fn_0\simeq \frak{so}(3,\R)$ spanned by $e_{ij}$, $i,j=1,2,3$.
\end{proof}
The computation from the proof of Lemma \ref{akce}. particularly implies that  
generators of $\frak{so}(3,\R)$ acts on $\frak n_{-2}$ as follows
$$ [e_{ij},e_{4+i}]
=[e_{ij},[e_0,e_i]]
=[[e_{ij},e_0],e_i]+[e_0,[e_{ij},e_i]]=[e_0,e_{j}] =e_{4+j}.
$$
The description of infinitesimal symmetries then follows.
\begin{prop*}
The Lie algebra of infinitesimal symmetries of $(N,\N,r)$ consists of right-invariant vector fields corresponding to $e_i$, $i=1,..7$ that generate all transvections on $N$ together with isotropy subalgebra isomorphic to $\frak{so}(3,\R)$.
\end{prop*}
\begin{proof}
The fact that all right--invariant vector fields determine infinitesimal symmetries follows from the fact that we deal with Lie group. Flows of right--invariant vector fields act as left translations and each right--invariant vector field is then an infinitesimal symmetry of any left--invariant object, \cite{CS}. The previous Lemma shows that the isotropy subalgebra of infinitesimal symmetries coincides with $\frak{so}(3,\R)$. Since the structure is of first order, the prolongation stops for $\fn_0$ and there cannot be symmetries of higher order, \cite{zel}.
\end{proof}

\begin{rem*}
In particular, the action given by $\frak{so}(3,\R)$ acts on Heisenberg subalgebras of $\fn$ from Remark \ref{rem1}, i.e. maps each such Heisenberg subalgebra to another Heisenberg subalgebra.
\end{rem*}

Let us note that above observations also imply that symmetries of the sub-Riemannian structure $(N,\N,r)$ preserve the decomposition $E+V:=\langle N_0 \rangle \oplus \langle N_1,N_2,N_3\rangle$. One can easily see that the decomposition satisfies
\begin{enumerate}
\item $E \cap V =0$, 
\item the Lie bracket of two sections of $V$ is a section of $E \oplus V$, and 
\item for sections $\xi \in \Gamma(E)$, $\nu \in \Gamma(V)$ and a point $q \in N$, the equation $[\xi,\nu] (q) \in E_q \oplus V_q$ implies $\xi(q)=0$ or $\nu(q)=0$.
\end{enumerate}
Geometric structures satisfying these three conditions are known as generalized path geometries, \cite[Section 4.4.3]{CS}, and correspond to parabolic geometries of type $(G,P)$ for $\fg=\frak{sl}(n,\R)$ and $\fp=\fp_{1,2}$ is the infinitesimal stabilizer of the flag of a line in a plane for the standard action.
In particular, they always have finite--dimensional Lie algebras of infinitesimal symmetries 
and maximum occurs for geometries that are locally equivalent to generalized flag manifold $G/P$ and equals to dim$(\fg)$.

It is not difficult to verify (e.g. by prolongation methods) that the Carnot group $N$ carries a maximally symmetric generalized path geometry with $\fg=\frak{sl}(5,\R)$ that particularly contains our $\frak{so}(3,\R)$ in $\fp_{1,2}$. This suggests us a way how to realize our situation using block $(1,1,3)$--matrices, \cite[Section 4.4.3]{CS}. The algebra $\frak{sl}(5,\R)$ carries a $|2|$--grading that inherits $\fn \oplus \frak{so}(3,\R)$ as follows
\begin{align} \label{algebry}
\begin{pmatrix}
\fg_0 & \fg_1 &\fg_2\\
\fn_{-1}^E & \fg_0 &  \fg_1\\
\fn_{-2} & \fn_{-1}^V & \fg_0 
\end{pmatrix}
\supset 
\begin{pmatrix}
0 & 0 &0\\
\fn_{-1}^E & 0 &  0\\
\fn_{-2} & \fn_{-1}^V & \fn_0 
\end{pmatrix},
\end{align}
and thus $\fn$ can be viewed as a choice of the complement of the stabilizer, i.e. a representative of the associated grading.

Having these matrices at hand, we can particularly view them as representatives of suitable exponential coordinates around the origin. Indeed, identifying coordinates of points around the origin with matrices
$$
\left( \begin{smallmatrix} 
0 & 0 &0\\
x & 0 &0\\
-y_i & \ell_i &0
\end{smallmatrix}
\right)
,$$
we recover the group structure \eqref{grupa} as 
\begin{align*} 
& \log \big ( \exp
\left( \begin{smallmatrix} 
0 & 0 &0\\
x & 0 &0\\
-y_i & \ell_i &0
\end{smallmatrix}
\right)
\exp
\left(
\begin{smallmatrix} 
0 & 0 &0\\
\tilde x & 0 &0\\
-\tilde y_i & \tilde \ell_i &0
\end{smallmatrix} 
\right)
\big )  =  
\log \big (
\left(
\begin{smallmatrix}  1&0&0\\ x&1&0
\\- y_i+\frac{1}{2} \ell_ix&\ell_i&1\end{smallmatrix}
\right) \left(
\begin{smallmatrix} 1&0&0\\ \tilde x&1&0
\\ - \tilde y_i+\frac{1}{2} \tilde \ell_i \tilde x& \tilde \ell_i &1
\end{smallmatrix} \right) \big
)
\\
&=\log \left(
\begin{smallmatrix} 1&0&0\\ x+\tilde x&1
&0\\ 
-y_i-
\tilde y_i+
\frac{1}{2} ( \ell_ix+
\ell_i \tilde x+
\tilde \ell_i  \tilde x )&
\ell_i+ \tilde \ell_i&1 
\end{smallmatrix} \right)
= \left( \begin{smallmatrix}  
0&0&0\\
x+\tilde x&0 &0 \\ 
-(y_i+{ \tilde y_i}  +\frac{1}{2}( 
 \tilde \ell_i x -\ell_i \tilde x))  
& \ell_i+{\tilde \ell_i}&0
\end{smallmatrix}
\right).
\end{align*}
This allows us to describe explicitly the action of isotropy symmetries in coordinates.
 \begin{prop*}
 \label{symmetries}
The action of isotropy symmetry $R \in SO(3,\R)$ on $ (x,\ell,y) \in \R^7 \simeq \R \oplus \R^3\oplus \R^3 $ takes form 
\begin{align}
\begin{split} 
(x,\ell,y) \mapsto (x, R \ell, Ry),
\end{split}
\label{action}
\end{align}
where we
 denote $\ell=(\ell_1,\ell_2,\ell_3)^t$ and $y=(y_1,y_2,y_3)^t$.
 \end{prop*}
\begin{proof}
Identifying $T_eN$ with $\fn$, the tangent action $T_oR$ reads as $X \mapsto \Ad_R(X)$ for $X \in \fn$. 
We see from  \eqref{algebry} that each element of $\frak{so}(3,\R)$ exponentiates to a matrix of the form
$$
\left( \begin{smallmatrix} 
1 & 0 &0\\
0 & 1 &0\\
0 & 0 &R
\end{smallmatrix} \right).
$$
Then we compute in coordinates
$$
\log \big( \left( \begin{smallmatrix} 
1 & 0 &0\\
0 & 1 &0\\
0 & 0 &R
\end{smallmatrix} \right)
\exp\left( \begin{smallmatrix} 
0 & 0 &0\\
x & 0 &0\\
-y & \ell &0
\end{smallmatrix} \right)
\left( \begin{smallmatrix} 
1 & 0 &0\\
0 & 1 &0\\
0 & 0 &R^{-1}
\end{smallmatrix} \right) \big)=
\left( \begin{smallmatrix} 
0 & 0 &0\\
x & 0 &0\\
-Ry & R\ell&0
\end{smallmatrix} \right)  
$$
and the formula follows.
\end{proof}

One can see from \eqref{action} that the action of the of $SO(3,\R)$ 
is given by
simultaneous rotations on $\ell_i$ and $y_i$, $i=1,2,3$, while the coordinate $x$ is invariant.  Since all invariants of each rotation in $\R^3$ are multiples of its axis, the fixed points of the symmetry with the axis $(a_1,a_2,a_3)$ form the set
\begin{align} \label{fp}
\{(x, k a_1,k a_2,k a_3,l a_1,l a_2,l a_3): x,k,l \in \R \}.
\end{align}
Finally, there is the following consequence of Proposition \ref{symmetries}.

\begin{cor*}
Set of points that are fixed by some isotropy symmetry is the union of sets \eqref{fp} over all axes $(a_1,a_2,a_3)$ 
\begin{align} \label{cn}
C_\fn = \{(x,\ell,y) \in \R\times\R^3\times\R^3:
\ell \text{ and } y \text{ are linearly dependent}
\}.
\end{align}
\label{corsym2}
\end{cor*}
Let us emphasize that $C_\fn$ is invariant with respect to the action of $SO(3,\R)$ on $N$. Moreover, for any fixed $\ell \in \R^3$, the set $\{ (x,y\ell,t\ell) : (x,y,t) \in \R^3 \} $ is a subgroup.

\section{Local control and geodesics} \label{sec-control}
Let us now focus on the control system (\ref{control},\ref{energy}) related to the sub--Riemannian structure $(N,\N,g)$. 
We use Hamiltonian concepts and we follow here \cite[Sections 7 and 13]{ABB} to find local control for the system. 
Left--invariant vector fields $N_0$, $N_1$,  $N_2$,  $N_3$, 
$N_{01}$, $N_{02}$, $N_{03}$ form a basis of $TN$ and determine left--invariant coordinates on $N$. The corresponding left--invariant coordinates $h_0,h_i$ and $w_i$, $i=1,2,3$ on fibers of $T^*N$ are given by $h_0 (\lambda )=  \lambda(N_0), h_i (\lambda )=  \lambda(N_i) $, 
$w_i (\lambda )=  \lambda(N_{0i})$
for arbitrary $1$--form $\lambda$ on $N$. Thus we can use $(x,\ell_i,y_i,h_0,h_i,w_i)$ as global coordinates on $T^*N$. Then in these coordinates, the corresponding Pontryagin's maximum principle system is as follows. Let us emphasize that geodesics, i.e. admissible curves parametrized by constant speed whose 
 sufficiently small arcs are length minimizers, are exactly projections on $N$ of solutions of this system, \cite{ABB}.

Firstly, we get $\dot w_i=0$ and thus 
$w_i$ are constant for $i=1,2,3$, i.e. we have 
\begin{align} \label{h567}
w_1=K_1,\ \ w_2=K_2,\ \ w_3=K_3
\end{align} 
for suitable constants $K_i$. Then for $h=(h_0,h_1,h_2,h_3)^t$ we get
$
\dot h = - \Omega h
$
for
\begin{align}
\Omega=
\left( \begin{smallmatrix} 
0 & K_1 & K_2 & K_3\\
-K_1 & 0 & 0 & 0\\
-K_2 & 0 & 0 & 0 \\
-K_3 & 0 & 0 & 0
\end{smallmatrix} \right).
\label{matrix}
\end{align}
Solution of the system is given by  
$
h(t)= e^{-t \Omega} h(0)
$, where  $h(0)$ is the initial value of the vector $h$ in the origin. If $K_1=K_2=K_3=0$, then $h(t)=h(0)$ is constant and the geodesic   $(x(t),\ell_i(t),y_i(t))$ is a line in $N$ such that $y_i(t)=0$.   
In next we assume that the vector $(K_1,K_2,K_3)$ is non--zero and 
we denote by  $K=\sqrt{K_1^2+K_2^2+K_3^2}$ its length.

\begin{prop*} \label{vert-p} The general solutions of $\dot h=\Omega h$ satisfying \eqref{h567} for non--zero $K$ take form   
\begin{align} 
\begin{split}\label{h1}
h_0&=K(-C_1\sin(Kt)+C_2\cos(Kt)), \\ 
\left( \begin{smallmatrix} 
h_1\\
h_2\\
h_3 
\end{smallmatrix} \right)&=
(C_1\cos(Kt)+C_2\sin(Kt))
\left( \begin{smallmatrix} 
K_1\\
K_2\\
K_3
\end{smallmatrix} \right)
+\left( \begin{smallmatrix} 
-C_3K_3-C_4K_2\\
C_4K_1\\
C_3K_1
\end{smallmatrix} \right)
\end{split} 
\end{align}
where $C_1,C_2,C_3,C_4$ are real constants. 
\end{prop*}
\begin{proof}

The solution of the system  is given by exponential of the matrix $\Omega$ from  \eqref{matrix}.
We need to  analyze its eigenvalues and eigenvectors.  
It follows that there are (complex conjugated) imaginary eigenvalues $\pm iK$  both of multiplicity one and the eigenvalue $0$ of multiplicity two. The corresponding eigenspace of $iK$ is generated by complex eigenvector $v$ that decomposes into real and complex component as
$\Re(v)=\left( 
0 ,
K_1,
K_2 ,
K_3
\right)^t$ and $ 
\Im(v)=\left(  
K,
0,
0,
0
\right)^t
.$
In the basis formed by these two vectors together with any basis of the  
two--dimensional eigenspace corresponding
to the eigenvalue $0$, the matrix $\Omega$ has zeros at all positions except positions  $(\Omega)_{12}=-(\Omega)_{21}=K$. Then we get in this eigenvector basis    
\begin{align}
e^{-t\Omega}  =
\left( \begin{smallmatrix}
\cos(Kt) & \sin(Kt) &  0 & 0 \\
 -\sin(Kt) & \cos(Kt) &  0 & 0 \\
 0&0& 1 &0 \\
 0& 0& 0&1
\end{smallmatrix} \right) .
\end{align}
For the choice of the basis of the eigenspace corresponding
to the eigenvalue $0$ given as $\left( 0 ,-K_3,0 ,K_1 \right)^t$ and 
$\left( 0 ,-K_2,K_1,0 \right)^t,$
the solution can be written as the combination 
\begin{align*}
\left( \begin{smallmatrix} 
h_0\\
h_1\\
h_2\\
h_3 
\end{smallmatrix} \right)&
=
(C_1\cos(Kt)+C_2\sin(Kt))
\left( \begin{smallmatrix} 
0 \\
K_1\\
K_2\\
K_3
\end{smallmatrix} \right)+
(-C_1\sin(Kt)+C_2\cos(Kt))\left( \begin{smallmatrix} 
K \\
0\\
0\\
0
\end{smallmatrix} \right)
\\
&+C_3\left( \begin{smallmatrix} 
0 \\
-K_3\\
0\\
K_1
\end{smallmatrix} \right)+
C_4\left( \begin{smallmatrix} 
0 \\
-K_2\\
K_1\\
0
\end{smallmatrix} \right)
\end{align*}
with coefficients $C_1,C_2,C_3,C_4$. Then the formula $(\ref{h1})$ follows.
\end{proof}

Let us emphasize that the choice $C_1=C_2=0$ gives constant solutions that are not relevant as control functions. Thus we assume that at least one of the constants $C_1,C_2$ is non--zero.

The base system for $x,y,\ell$ then takes the explicit form
\begin{align}
\begin{split} \label{xl}
\dot x &= h_0,\ \   \dot \ell_1 = h_1, \ \
\dot \ell_2 = h_2, \ \  \dot \ell_3 = h_3,\\
\dot y_1&= {1 \over 2}(xh_1-h_0\ell_1), \ \  \dot y_2= {1 \over 2}(xh_2-h_0\ell_2),\ \  \dot y_3= {1 \over 2}(xh_3-h_0\ell_3).
\end{split}
\end{align}
We are interested in solutions emanating from the origin, i.e. we impose the initial condition $x(0)=0,\ell_i(0)=0,y_i(0)=0$, $i=1,2,3$. Indeed, we can find geodesics starting at different point of $N$ using the action of suitable transvection, see Proposition \ref{symmetries}.

\begin{prop*} \label{geodesics}
Arc--length sub--Riemannian geodesics 
on Carnot group $N$ satisfying the initial condition $x(0)=0,\ell_i(0)=0,y_i(0)=0$, $i=1,2,3$  are either lines of the form 
 \begin{align} 
 (x, \ell_1,\ell_2,\ell_3,y_1,y_2,y_3 )^t =  
 (C_1t , C_2 t,C_3 t,C_4 t,0,0,0)^t
 \label{line} 
 \end{align}
 parametrized by constants $C_1, C_2,C_3, C_4$ satisfying $C_1^2+C_2^2+C_3^2+C_4^2=1$, or they are curves given by equations 
\begin{align} \label{x}
x&=C_1\cos(Kt)+C_2\sin(Kt)-C_1,
\\
\label{xl2}
\left( \begin{smallmatrix} 
\ell_1\\
\ell_2\\
\ell_3
\end{smallmatrix} \right)& =
{1 \over K}
(C_1\sin(Kt)-C_2\cos(Kt)+C_2)
\left( \begin{smallmatrix} 
K_1\\
K_2\\
K_3
\end{smallmatrix} \right)
+t\left( \begin{smallmatrix} 
-C_3K_3-C_4K_2\\
C_4K_1\\
C_3K_1
\end{smallmatrix} \right), 
\\
\begin{split} 
\label{y}
\left( \begin{smallmatrix} 
y_1\\
y_2\\
y_3
\end{smallmatrix}\right)&={1 \over 2K}(C_1^2+C_2^2)(tK-\sin(Kt))\left( \begin{smallmatrix} 
K_1\\
K_2\\
K_3
\end{smallmatrix} \right)+
{1 \over 2K}((2C_1-C_2Kt)\sin(Kt)\\ &-(C_1Kt+2C_2 )\cos(Kt) +2C_2-tC_1K)\left( \begin{smallmatrix} 
-C_3K_3-C_4K_2\\
C_4K_1\\
C_3K_1
\end{smallmatrix} \right)
,
\end{split} 
\end{align}
 parameterized by constants $C_1, C_2,C_3, C_4$ and  $K_1, K_2,K_3$ satisfying
\begin{align} 
\begin{split} \label{level}
K^2(C_1^2+ C_2^2) +(C_3K_3+C_4K_2)^2+C_4^2K_1^2+ C_3^2K_1^2 =1
\end{split}.
\end{align} 

\label{geo}
\end{prop*}
\begin{proof}
Let us firstly remind that parametrization of geodesics is encoded in level sets of the Hamiltonian of the system (\ref{control},\ref{energy}) that is $H=\frac12(h_0^2+h_1^2+h_2^2+h_3^2)$ and arc length parametrization correspond to $H=\frac12$, \cite{ABB}.

The line \eqref{line} corresponds to $K=0$ and thus $h(t)=h(0)$ is constant  and  defines the vector of constants $(C_1,C_2,C_3,C_4)$. The length of this vector is equal to one on the level set $\frac{1}{2}$.  If $K \neq 0$ we obtain $x, \ell_1,\ell_2,\ell_3$ by direct integration of the first part of \eqref{xl} and involving the initial condition, where $h(t)$ is given by \eqref{h1}. 
Substituting the results into the second part of \eqref{xl} we get  $y_1,y_2,y_3$ by integration. 
The solutions define family of curves starting at the origin such that  the Hamiltonian $H$ is constant along them. 
Unit--speed geodesics are contained in
the level set $h_0^2+h_1^2+h_2^2+h_3^2=1$. According to Proposition \ref{vert-p} this restriction reads  as  \eqref{level}.
\end{proof}

\section {Moduli space and geodesics } \label{sec4}

Each choice of coefficients $C_1,C_2,C_3,C_4 \in \mathbb R$ and $K_1,K_2,K_3 \in \mathbb R $ 
that satisfy \eqref{level} gives a geodesic $(x(t),\ell(t),y(t))$ as described in the Proposition \ref{geo}. According to \eqref{xl} and \eqref{y}, 
$\ell(t)$ and $y(t)$ are linear combinations of the vectors
$$
z_1=\left( \begin{smallmatrix} 
K_1\\
K_2\\
K_3
\end{smallmatrix} \right)
,\ \ \ z_2=\left( \begin{smallmatrix} 
-C_3K_3-C_4K_2\\
C_4K_1\\
C_3K_1
\end{smallmatrix} \right)
$$
for any $t >0$. The vectors 
$z_1$ and $z_2$ are orthogonal with respect to the Euclidean metric on $\mathbb R^3$ by definition. 
We know from Proposition \ref{symmetries} that for each $R \in SO(3,\R)$ the map 
\begin{align*}
\begin{split} 
(x,\ell,y) \mapsto (x, R \ell, Ry)
\end{split}
\end{align*}
maps geodesics starting at the origin to geodesics starting at the origin.

Altogether, there always is an orthogonal matrix $R \in SO(3,\R)$ that aligns vectors $z_1$ and $z_2$  with the suitable multiples of the first two vectors of the standard basis of $\R^3$. Thus we get
$$ z_1 = R
\left( \begin{smallmatrix} 
K\\
0\\
0
\end{smallmatrix} \right)
,\ \ \ z_2 = R\left( \begin{smallmatrix} 
0\\
C\\
0
\end{smallmatrix} \right), 
$$
where $K=|z_1|$ is the length of $z_1$ and we denote  $C=|z_2|$ the length of $z_2$.
This matrix defines a representative of the geodesic class 
$$(x(t),\bar \ell(t),\bar y(t)) = (x(t),R^t\ell(t),R^ty(t)).$$
The equations for this representative geodesics simplify remarkably. Namely 
\begin{align}
\begin{split} 
x(\tau) &= C_1 (\cos \tau - 1) + C_2 \sin \tau, \\
\bar \ell_1(\tau) &= C_1 \sin \tau + C_2 (1-\cos \tau ), \\
\bar \ell_2 (\tau)&= \bar C_3 \tau, \\
\bar \ell_3 (\tau)&= 0, \\
\bar y_1 (\tau)&= \frac12 (C_1^2+C_2^2)(\tau - \sin \tau), \\
\bar y_2 (\tau)&= \frac12 \bar C_3 \left[C_1 (2 \sin \tau  - \tau \cos \tau - \tau   )+ 
C_2 (2-2 \cos \tau  - \tau \sin \tau  ) \right], \\
\bar y_3 (\tau)&= 0,
\end{split}
\label{newgeo}
\end{align}
where $\tau=Kt$ and $\bar C_3 = C/K$.
The level set equation \eqref{level} reads as 
\begin{align} 
K^2(C_1^2+C_2^2+\bar C_3^2) = 1
\label{kc}
\end{align}
and determines $K>0$ uniquely.

The moduli space $N/SO(3,\R)$ defined by the action \eqref{action} of $SO(3,\R)$ on $N \cong \mathbb R^7$ is determined by natural invariants 
$ x, (\ell,\ell), (\ell,y), (y,y),$ 
where $(\:,\:)$ stands for the Euclidean scalar product on $\mathbb R^3$. 
\begin{prop*} \label{prop4}
Each geodesic starting at the origin defines a curve in the moduli space $N/SO(3,\R)$ 
given by a curve in invariants 
\begin{align}
\begin{split}
x &= C_1 (\cos \tau - 1) + C_2 \sin \tau, \\
(\ell,\ell) &= (C_1 \sin \tau + C_2 (1-\cos \tau ))^2+ (\bar C_3 \tau)^2, \\
(\ell,y) &=\frac12 (C_1^2+C_2^2)
(C_1 \sin \tau + C_2 (1-\cos \tau ) )(\tau - \cos \tau) \\
&+\bar C_3^2 \tau  [C_1 (2 \sin \tau  - \tau \cos \tau - \tau   )+ 
C_2 (2-2 \cos \tau  - \tau \sin \tau  ) ],\\
(y,y) &=\frac14 (C_1^2+C_2^2)^2(\tau - \cos \tau)^2+ \bar C_3^2 [C_1 (2 \sin \tau  - \tau \cos \tau - \tau   )\\
&+ C_2 (2-2 \cos \tau  - \tau \sin \tau  ) ]^2.
\end{split}
\label{inv} 
\end{align}
\label{invp}
\end{prop*}
\begin{proof}
Follows directly from \eqref{newgeo}.
\end{proof}

Let us recall that the subgroup $C_{\fn}$, defined by \eqref{cn}, consists of points in $N$ that are stabilized by some non--trivial $R \in SO(3,\R)$ for the action \eqref{action}.
Note the similarity of $C_\fn$ to the set $P_3$ from nilpotent $(3,6)$ sub--Riemannian problem which is known to be the set where geodesics starting at the origin lose optimality, \cite{mya1}. For any point of $P_3$ there exists a one--parameter family of geodesics of equal length intersecting at this point.  However, the situation in our nilpotent $(4,7)$ problem is very different.

\begin{thm*} \label{thm1}
	Sub--Riemannian geodesics starting at the origin either do not intersect $C_\fn$ or they lie in $C_\fn$ for all times. 
\end{thm*}

\begin{proof}
Suppose there is an intersection of the set $C_\fn$ with a sub--Riemannian geodesic $(x(t),\ell(t),y(t))$ emanating from the origin.  
So there is a point of intersection  $(x,\bar \ell,\bar y)$ of the set $C_\fn$ with a sub--Riemannian geodesic $(x(t),\bar \ell(t),\bar y(t))$ since 
 $C_{\mathfrak  n}$  is invariant with respect to the action \eqref{action} of $SO(3,
R)$. 
At this intersection  $(x,\bar \ell,\bar y)$, the collinearity of $\bar \ell$ and $\bar y$ is described by vanishing of the determinant 
$$\bar \ell_1 \bar y_2 - \bar \ell_2 \bar y_1  = 0.$$
The  geodesics are given by equations \eqref{newgeo} and the determinant can be written explicitly as $\frac12$--multiple of  
\begin{align}
\bar C_3 (d_{11}C_1^2+2d_{12}C_1C_2+d_{22} C_2^2),
\label{det} 
\end{align}
where 
\begin{align*}
d_{11} &=-{\tau}^{2}-\tau \sin \tau \cos  \tau
		 -2\, \cos^2  \tau +2 ,\\
d_{12} &=-2\sin  \tau \left( 2\,\cos \tau -2
		+\tau \sin \tau  \right), \\
d_{22} & = 2\cos^2 \tau + 
   \tau  \cos \tau \sin \tau  -4  \cos \tau   -{\tau}^{2}+2.
\end{align*}
We show 
that the function in the bracket of \eqref{det} is never zero (unless $C_1=C_2=0$, which is irrelevant)  by showing that its discriminant $d$ of this quadratic equation is negative for all positive times. This implies that the colinearity condition \eqref{det} is equivalent to  $\bar C_3=0$. Then $\bar \ell_2 (\tau)=\bar y_2(\tau) =0$ by \eqref{newgeo} and thus geodesic  
$(x(\tau),\bar \ell(\tau),\bar y(\tau))$ belongs to $C_{\fn}$ for all $\tau >0$. 

To show that the discriminant $d$  is negative for all positive times we compute 
\begin{align*}
d = -4 \tau(\tau-\sin\tau)(\tau^2+\tau\sin\tau+4\cos\tau-4),
\end{align*}
hence it is sufficient to prove 
\begin{align*}
f(\tau) = \tau^2+\tau\sin\tau+4\cos\tau-4 >0
\end{align*}
for all positive times $\tau$. This can be done by combining "local" and "global" estimations of this function. The local estimation is obtained by the estimation of goniometric functions $\sin\tau, \cos\tau$ by Taylor series. By evaluating the Taylor series of $f$ in zero, we see we need to use the Taylor polynomial of degree seven and six, respectively.
Then we get the estimation 
\begin{align*}
f(\tau)>-\frac{\tau^6(\tau^2-14)}{5040}
\end{align*}
that guarantees positivity of the function $f$ on the interval $(0,\sqrt{14})$. On the other hand, the inequalities $\sin\tau,\cos\tau\geq -1$ yield a global estimation 
 \begin{align*}
 f(\tau)>\tau^2-\tau-8
 \end{align*}
 that guarantees positivity of the function $f$ on the interval $(\tfrac{1+\sqrt{33}}{2},\infty)$. The two intervals overlap and thus $f$ is positive for all $\tau>0$.
 \end{proof}
In the Figure \eqref{omez} we present both local and global estimation of $f(\tau)$.
 \begin{figure}[h]
 \begin{center}
 \includegraphics[height=40mm]{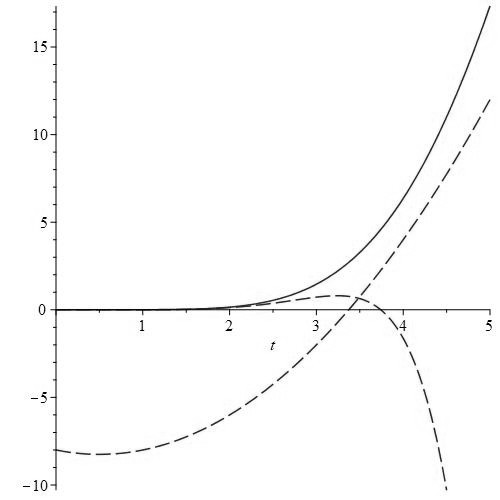}
 \caption{Estimation of $f(\tau)$}
 \label{omez}
 \end{center}
 \end{figure}

\begin{rem*}
The positivity of function $f(\tau)$ from the proof above can be shown alternatively  by proving the positivity of its derivative. The alternative proof can be found in Lemma 3.1. for $\tau=2\theta$ of \cite{talianky}, where authors discuss free $2$--step Carnot group of filtration $(3,6)$.
\end{rem*} 

Let us finally study properties of geodesics contained in $C_{\fn }$.
According to \eqref{det}, this happens if and only if  $\bar C_3=0$. 
Then the non--zero parts of geodesics \eqref{newgeo} are   
\begin{align}
\begin{split} 
x(t) &= C_1 (\cos (Kt) - 1) + C_2 \sin (Kt), \\
\bar \ell_1(t) &= C_1 \sin (Kt) + C_2 (1-\cos (Kt)), \\
\bar y_1 (t)&= (C_1^2+C_2^2)((Kt) - \sin(Kt)), \\
\end{split}
\label{newgeoH}
\end{align}
and level set condition \eqref{level} reads as
$$ {K^2}(C_1^2+C_2^2)= 1.$$
We show that these geodesics are preimages of   geodesics in   Heisenberg geometry and their optimality is well known, \cite{ABB,hzn, mope}. 
Thus, we get the following statement.

\begin{thm*} \label{thm2}
The vertical set $ \{ (0,0,y) \in C_{\fn} :  y \in \mathbb R^3 \} $ is the set where the geodesics in $C_{\fn}$ starting at the origin lose their optimality. These points are Maxwell points and 
for geodesics defined by parameters $C_1$ and $C_2$    
the cut time is $$t_{cut} = 2\pi \sqrt{ C_1^2+C_2^2}.$$
\end{thm*}
\begin{proof} 
Since $\bar \ell_1 = \sqrt{(\ell,\ell)}$ and $\bar y_1 = \sqrt{(y,y)}$ are invariants, see \eqref{inv}, the expression \eqref{newgeoH} defines a curve in the factor space $C_{\fn}/SO(3,\R)$.  
For the choice of polar coordinates in the plane $\langle C_1,C_2 \rangle$ 
we get the standard description of geodesics on three--dimensional Heisenberg group $\mathbb H_3$. Indeed, the tangent space to the subgroup $C_{\fn}$
is generated by pushout vectors 
\begin{align*}
\begin{split}
\bar N_0=\partial_{x}-{\bar \ell_1 \over 2}{\partial_{\bar y_1}},\ \ \ \ 
\bar N_1=\partial_{\bar \ell_1}+{x \over 2} \partial_{\bar y_1}
\end{split}
\end{align*}
that are standard generators of Heisenberg Lie algebra. 
The group law 
\eqref{grupa} on the subgroup $C_{\fn}$ defines an isomorphism  $C_{\fn}/SO(3,\R) \cong \mathbb H_3$.
 The cut locus of the Heisenberg group consists of the
set of points $$\{(0, y ) \in \mathbb R^2 \oplus \wedge^2 \mathbb R^2 :  y \neq 0  \}.$$  Namely, 
any geodesic from the origin loses its optimality at the point where it meets the vertical axis $(0, 0, \bar y)$ for the first time. These points are Maxwell points and the corresponding time equals to $t_{cut} = \frac{2\pi}{K}$. 
Sub--Riemannian geodesics in $C_{\fn}$ going from the origin to the point 
$(x,\ell,\lambda \ell)$ form a preimage of the Heisenberg
geodesic in $C_{\fn}/SO(3,\R)$ going from the origin to the point 
$(x,|\ell|,\lambda |\ell|)$, where $|\ell|^2=(\ell,\ell)$. These geodesics have the same length and they lose their optimality at the same time. 
\end{proof} 

Since the geodesics contained in $C_{\fn}$ are preimages of Heisenberg geodesics under the  $SO(3,\R)$--action, we can visualize them in the same way.  On the left hand side of Figure \ref{HG}, there is  so--called Heisenberg sub--Riemannian sphere. On the right side of the same figure, there is a half--sphere with a family of geodesics from origin to the sphere.     
 \begin{figure}[h]
 \begin{center}
 \includegraphics[height=50mm]{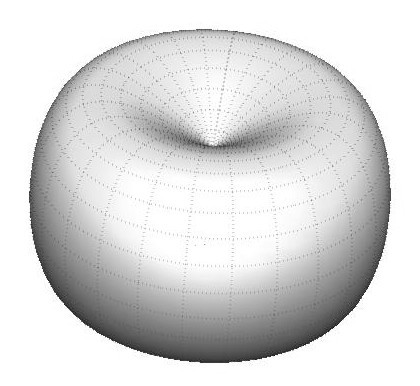}
 \includegraphics[height=50mm]{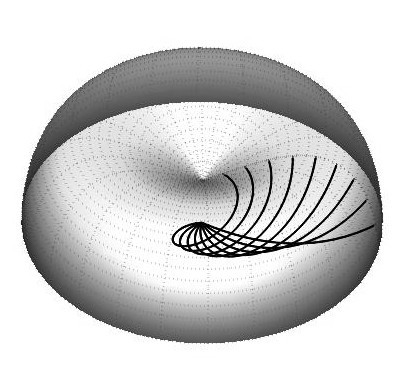}
 \caption{Heisenberg sub--Riemannian sphere and geodesics from origin to the 
 points of the sphere 
 }
 \label{HG}
 \end{center}
 \end{figure}

\section{Declarations}
Conflict of interest: The authors declare that they have no conflict of interest.

\end{document}